\documentclass[a4paper,12pt]{article}

\usepackage{verbatim}
\usepackage{amsmath}
\usepackage{amssymb}
\usepackage{amsthm}
\usepackage{graphicx}
\usepackage[all]{xy}

\newtheorem{thm}{Theorem}[section]
\newtheorem{theorem}[thm]{Theorem}
\newtheorem{prop}[thm]{Proposition}

\newtheorem{cor}[thm]{Corollary}

\newtheorem{lemma}[thm]{Lemma}
\newtheorem{problem}[thm]{Problem}

\theoremstyle{definition}

\newtheorem{remark}[thm]{Remark}

\newcommand{\lmd}[1]{\hbox to2.7em{$ \lambda_{#1} \hfill$}}
\newcommand{\tlmd}[1]{\hbox to2.7em{$ |\lambda_{#1}| \hfill$}}
\newcommand{\llmd}[2]{\hbox to2.7em{$\lambda_{#2}^{(#1)} \hfill$ }}
\newcommand{\vlmd}[2]{\hbox to2.7em{$ |\lambda_{#2}^{(#1)}| \hfill $}}

\newcommand{\bxcdots}{\hbox to2.7em{$ \cdots \hfill $}}
\newcommand{\bxvdots}{\hbox to2.7em{$ \vdots \hfill $}}
\newcommand{\bxddots}{\hbox to2.7em{$ \ddots \hfill $}}

\newcommand{\re}{\mathrm{Re}\,}
\newcommand{\im}{\mathrm{Im}\,}

\newcommand{\reg}{\mathrm{reg}}

\newcommand{\gammatilde}{{\widetilde{\gamma}}}

\newcommand{\bC}{\mathbb{C}}

\newcommand{\bP}{\mathbb{P}}

\newcommand{\bR}{\mathbb{R}}

\newcommand{\frakm}{\mathfrak{m}}
\newcommand{\frakX}{\mathfrak{X}}

\newcommand{\scM}{\mathcal{M}}
\newcommand{\scO}{\mathcal{O}}

\newcommand{\PD}{\operatorname{PD}}
\newcommand{\Int}{\operatorname{Int}}

\newcommand{\po}{\mathfrak{PO}}

\title{Potential functions via toric degenerations}
\author{Takeo Nishinou, Yuichi Nohara, and Kazushi Ueda}
\date{}

\begin{document}

\maketitle

\begin{abstract}
This is a short companion paper to \cite{Nishinou-Nohara-Ueda_TDGCSPF}.
We construct an integrable system
on an open subset of a Fano manifold equipped with a toric degeneration,
and compute the potential function for its Lagrangian torus fibers
if the central fiber is a toric Fano variety
admitting a small resolution.
\end{abstract}

\section{Introduction}

An {\em integrable system}
is a set $\{ \Phi_i \}_{i=1}^N$ of $N$ functions
on a symplectic manifold $(M, \omega)$
of dimension $2 N$
which are functionally independent
and mutually Poisson-commutative;
$$
 \{ \Phi_i, \Phi_j \} = 0, \qquad i, j = 1, \dots, N.
$$
Here, a set $\{ \Phi_i \}_{i=1}^N$ of functions
on a manifold $M$ is said to be {\em functionally independent}
if there is an open dense subset $U \subset M$
where the differentials $\{ d \Phi_i \}_{i=1}^N$ are linearly independent.
An integrable system
$\{ \Phi_i \}_{i=1}^N$ defines a Hamiltonian $\bR^N$ action on $M$,
and any regular compact connected fiber of
$\Phi = (\Phi_1, \dots, \Phi_N) : M \to \bR^N$
is a torus by the Arnold-Liouville theorem.

For a Lagrangian submanifold $L$ in a symplectic manifold,
the cohomology group $H^*(L; \Lambda_0)$
with coefficient in the Novikov ring $\Lambda_0$
has a structure of a weak $A_\infty$-algebra
by the fundamental work of Fukaya, Oh, Ohta and Ono
\cite{Fukaya-Oh-Ohta-Ono}.
A solution to the Maurer-Cartan equation
$$
 \sum_{k=0}^\infty \frakm_k(b, \dots, b) \equiv 0
  \mod \PD([L])
$$
is called a weak bounding cochain,
which can be used to define the deformed Floer cohomology.
The potential function is a map
$
 \po : \scM(L) \to \Lambda_0
$
from the moduli space $\scM(L)$ of weak bounding cochains 
such that
$$
 \sum_{k=0}^\infty \frakm_k(b, \dots, b) = \po(b) \cdot \PD([L]).
$$

The moment map for the torus action
on a toric Fano manifold
with respect to a torus-invariant K\"{a}hler form
provides an example of an integrable system.
The potential function for its Lagrangian torus fiber
is computed by Cho and Oh \cite{Cho-Oh} and
Fukaya, Oh, Ohta and Ono \cite{FOOO_toric_I}.

We discuss the following problem in this paper:

\begin{problem} \label{prob:main}\ 
\begin{enumerate}
 \item
Which Fano manifold admits a structure of an integrable system?
 \item
Compute the potential function
for Lagrangian torus fibers in such cases.
\end{enumerate}
\end{problem}

Motivated by the work of Batyrev et al.
\cite{Batyrev_TD,
Batyrev-Ciocan-Fontanine-Kim-van_Straten_CT,
Batyrev-Ciocan-Fontanine-Kim-van_Straten_MSTD},
we assume that a Fano manifold $X$ has a toric degeneration,
so that there is a flat family
$$
 f : \frakX \to B
$$
of projective varieties over some base $B$ containing
two points $0$ and $1$
such that $X_t$ is smooth for general $t \in B$,
$X_1$ is isomorphic to $X$ and
the central fiber $X_0$ is a toric variety.
Choose a piecewise smooth path $\gamma : [0, 1] \to B$
such that $\gamma(0) = 0$, $\gamma(1) = 1$ and
$X_{\gamma(t)}$ is smooth for $t \in (0, 1]$.
Then the symplectic parallel transport
along $\gamma$ gives a symplectomorphism
$$
 \gammatilde : X_0^\reg \to X_1^\reg
$$
from the regular locus $X_0^\reg$ of $X_0$
to an open subset $X_1^\reg$ of $X_1$.
By transporting the toric integrable system
$$
 \Phi_0 : X_0 \to \bR^N
$$
to $X_1^\reg$ by $\gammatilde$,
one obtains an integrable system
$$
 \Phi = \Phi_0 \circ \gammatilde^{-1} : X_1^\reg \to \bR^N
$$
on $X_1^\reg$. Let
$$
 \ell_i(u) = \langle v_i, u \rangle - \tau_i
$$
be the affine functions
defining the faces of the moment polytope;
$$
 \Delta
  = \Phi_0(X_0)
  = \{ u \in \bR^N \mid \ell_i(u) \ge 0, \quad i = 1, \dots, m \}.
$$
The proof of \cite[Theorem 10.1]{Nishinou-Nohara-Ueda_TDGCSPF}
immediately gives the following:

\begin{theorem} \label{th:main}
Assume that $X_0$ is a Fano variety
admitting a small resolution.
Then for any $u \in \Int \Delta$,
one has an inclusion
$$
 H^1(L(u); \Lambda_0) \subset \scM(L(u))
$$
for the Lagrangian torus fiber $L(u) = \Phi^{-1}(u)$,
and the potential function is given by
$$
 \po(x) = \sum_{i=1}^m e^{\langle v_i, x \rangle} T^{\ell_i(u)}
$$
for $x \in H^1(L(u), \Lambda_0)$.
\end{theorem}

As a corollary,
one obtains
a non-displaceable Lagrangian torus
just as in \cite[Theorem 1.5]{FOOO_toric_I}:

\begin{cor}
If a Fano manifold $X$ admits a flat degeneration
into a toric Fano variety with a small resolution,
then there is a Lagrangian torus $L$ in $X$
satisfying
$$
 \psi(L) \cap L \ne \emptyset
$$
for any Hamiltonian diffeomorphism $\psi : X \to X$.
\end{cor}

The organization of this paper is as follows:
In Section \ref{sc:N2}, we illustrate Theorem \ref{th:main}
with an example of a toric degeneration.
In Section \ref{sc:quadric},
we discuss toric degenerations of quadric hypersurfaces
and their relation with Gelfand-Cetlin systems for orthogonal groups.
In Section \ref{sc:vc},
we show that the vanishing cycle on the quadric surface
with respect to the degeneration into the weighted projective plane
$\bP(1, 1, 2)$ is the image of the anti-diagonal Lagrangian submanifold
in $\bP^1 \times \bP^1$ by the Segre embedding.
In Section \ref{sc:f2}, we discuss an alternative approach
to compute the potential function for a Lagrangian torus in $S^2 \times S^2$
originally due to Auroux \cite{Auroux_SLFWCMS} and
Fukaya, Oh, Ohta and Ono \cite{FOOO_toric_deg},
which can also be applied to cubic surfaces.

%
%
%

\section{Complete intersection of two quadrics in $\bP^5$}
 \label{sc:N2}

It is known by Newstead \cite{Newstead}
and Narasimhan and Ramanan \cite{Narasimhan-Ramanan}
that the moduli space of stable rank two vector bundles
with a fixed determinant of odd degree
on a genus two curve defined as the double cover of $\bP^1$
branched over $\{ \omega_0, \dots, \omega_5 \} \subset \bC$
is a complete intersection $X = Q_1 \cap Q_2$ of two quadrics
$$
 Q_1 : \sum_{i=0}^{5} z_i^2 = 0
$$
and
$$
 Q_2 : \sum_{i=0}^{5} \omega_i z_i^2 = 0
$$
in $\bP^5$.
Since the total Chern class of the tangent bundle $T_X$ is
given by
\begin{align*}
 c(T_X)
  = \frac{c(T_{\bP^5}|_X)}{c(N_{X/\bP^5})}
  = \frac{(1 + \omega)^6}{(1 + 2 \omega)^2}
  = 1 + 2 \omega + 3 \omega^2,
\end{align*}
the top Chern class and hence the Euler number of $X$ vanishes.
Since the cohomology ring of a complete intersection is non-trivial
only at the middle dimension,
the cohomology group of $X$ has rank eight.
We equip $X$ with the K\"ahler form
$\omega = \lambda \omega_{\mathrm{FS}}|_X$
where $\lambda > 0$ and
$\omega_{\mathrm{FS}}$ is the Fubini-Study form.

Fano complete intersections in projective spaces
admit several toric degenerations in general.
In the case of two quadrics as above,
one possible degeneration is the complete intersection of
$$
 z_0 z_1 = z_2 z_3 \quad \text{and} \quad z_2 z_3 = z_4 z_5,
$$
which has a torus action given by
$$
 [z_0 : z_1 : z_2 : z_3 : z_4 : z_5]
  \mapsto
   [\alpha z_0 : \beta z_1 : \gamma z_2 : \alpha \beta \gamma^{-1} z_3
     : \alpha \beta z_4 : z_5],
$$
so that the moment polytope is the convex hull of
$$
 \{ (\lambda, 0, 0), (0, \lambda, 0), (0, 0, \lambda),
    (\lambda, \lambda, - \lambda), (\lambda, \lambda, 0), (0, 0, 0) \},
$$
which is an octahedron.
The central fiber has six ordinary double points
and admits a small resolution.
The defining inequalities for the moment polytope are
\begin{align*}
 \ell_1(u) &= \langle (0, 1, 1), u \rangle \ge 0, \\
 \ell_2(u) &= \langle (-1, 0, 0), u \rangle + \lambda \ge 0, \\
 \ell_3(u) &= \langle (0, -1, 0), u \rangle + \lambda \ge 0, \\
 \ell_4(u) &= \langle (1, 0, 1), u \rangle \ge 0, \\
 \ell_5(u) &= \langle (0, 1, 0), u \rangle \ge 0, \\
 \ell_6(u) &= \langle (-1, 0, -1), u \rangle + \lambda \ge 0, \\
 \ell_7(u) &= \langle (0, -1, -1), u \rangle + \lambda \ge 0, \\
 \ell_8(u) &= \langle (1, 0, 0), u \rangle \ge 0,
\end{align*}
so that the potential function is given by
\begin{align*}
 \po
  &= e^{x_2 + x_3} T^{u_2 + u_3}
      + e^{- x_1} T^{- u_1 + \lambda}
      + e^{- x_2} T^{- u_2 + \lambda}
      + e^{x_1 + x_3} T^{u_1 + u_3}
      + e^{x_2} T^{u_2} \\
  & \qquad + e^{- x_1 - x_3} T^{- u_1 - u_3 + \lambda}
      + e^{- x_2 - x_3} T^{- u_2 - u_3 + \lambda}
      + e^{x_1} T^{u_1} \\
  &= y_2 y_3 + \frac{Q}{y_1} + \frac{Q}{y_2} + y_1 y_3 + y_2
      + \frac{Q}{y_1 y_3} + \frac{Q}{y_2 y_3} + y_1.
\end{align*}
By equating the partial derivatives
\begin{align*}
 \frac{\partial \po}{\partial y_1}
  &= - \frac{Q}{y_1^2} + y_3 - \frac{Q}{y_1^2 y_3} + 1, \\
 \frac{\partial \po}{\partial y_2}
  &= y_3 - \frac{Q}{y_2^2} + 1 - \frac{Q}{y_2^2 y_3}, \\
 \frac{\partial \po}{\partial y_3}
  &= y_2 + y_1 - \frac{Q}{y_1 y_3^2} - \frac{Q}{y_2 y_3^2}
\end{align*}
with zero,
one obtains non-isolated critical points
defined by
\begin{align*}
  y_2 = - y_1, \quad &\text{and} \quad y_3 = - 1, \\
  y_2 = - y_1, \quad &\text{and} \quad y_3 = \frac{Q}{y_1^2},
 \qquad \text{or} \\
  y_2 = \frac{Q}{y_1}, \ \, \quad &\text{and} \quad y_3 = -1,
\end{align*}
and four non-degenerate critical points
defined by
\begin{align*}
 y_1^4 = Q^2, \qquad
 y_2 = \frac{y_1^3}{Q}, \qquad
 y_3 = \frac{y_1^2}{Q}.
\end{align*}
The valuations of the latter four critical points
are given by
$$
 (u_1, u_2, u_3) = (\lambda / 2, \lambda / 2, 0),
$$
and lies in the interior of the moment polytope.
The existence of non-isolated critical points
whose valuations lie in the moment polytope implies
the existence of a continuum of
non-displaceable Lagrangian tori.
See Fukaya, Oh, Ohta and Ono \cite[Theorem 1.1]{FOOO_toric_II}
for toric examples.

\section{Quadric hypersurfaces and Gelfand-Cetlin systems}
 \label{sc:quadric}

Let 
\[
  X = \{ [x_1:x_2:\dots:x_n] \in \bP^{n-1} \, | \, 
   x_1^2 + x_2^2 + \dots + x_n^2 = 0 \}
\]
be a quadric hypersurfaces in $\bP^{n-1}$ equipped with
a K\"ahler form
$\omega = \lambda \omega_{\mathrm{FS}}|_X$ ($\lambda > 0$).
In this case, we have two families
$$
 \frakX^3 = \{ ([x_1 : \cdots : x_n], t) \in \bP^{n-1} \times \bC
  \mid x_1^2 + x_2^2 + x_3^2 + t(x_4^2 \dots + x_{n+1}^2) = 0 \}
$$
and
$$
 \frakX^4 = \{ ([x_1 : \cdots : x_n], t) \in \bP^{n-1} \times \bC
  \mid x_1^2 + \dots + x_4^2 + t (x_5^2 + \dots + x_n^2) = 0 \}
$$
whose general fibers are isomorphic to $X$
and special fibers $X_0^3$ and $X_0^4$ are toric Fano varieties.
The quadric hypersurface $X_0^4$ of rank four has a small resolution,
so that one can apply Theorem \ref{th:main} to compute the potential function.
When $n = 4$, $X^4_0 = X$ is $\bP^1 \times \bP^1$
and $X^3_0$ is the weighted projective plane $\bP (1,1,2)$,
or equivalently,
the blow-down of the $(-2)$-curve
in the Hirzebruch surface
$F_2 = \bP(\scO_{\bP^1} \oplus \scO_{\bP^1}(2))$.

Let us discuss the degeneration $\frakX^3$ in more detail.
Since the central fiber is isomorphic to
$$
 {X^3_0}' = \{ [x_1':\dots:x_n'] \in \bP^n \mid 2 x_1' x_2' + x_3'^2 = 0 \}
$$
with the torus action given by
\[
  [x_1':x_2':\dots:x_n'] \longmapsto
  [\tau_2^{-1}\tau_3^2 x_1': \tau_2 x_2': \tau_3 x_3':
   \dots :\tau_n x_n']
\]
for 
$(\tau_2, \tau_3, \dots , \tau_n) \in 
(S^1)^{n-1}/(\text{diagonal})$,
its moment map
\[
  \mu_{T^{n-2}} : 
  \bP^{n-1} \to \{ (u_2, u_3, \dots, u_n) \in \bR^{n-1} 
  \,| \,
  \sum u_i = 1 \}
  \cong \bR^{n-2},
\]
with respect to $\lambda \omega_{\mathrm{FS}}$ is given by
\[
  \mu_{T^{n-2}} (x')
  = \lambda \left( 
  \frac{-|x_1'|^2+|x_2'|^2}{\|x' \|^2},
  \frac{2|x_1'|^2+|x_3'|^2}{\|x' \|^2},
  \frac{|x_4'|^2}{\|x' \|^2},
  \dots,
  \frac{|x_n'|^2}{\|x'\|^2}
  \right),
\]
where $\|x'\|^2 = \sum_{i=1}^n |x_i'|^2$.
Since the isomorphism between $X_0^3$ and ${X_0^3}'$ 
is given by
\[
  \begin{pmatrix}
   x_1' \\ x_2' \\ x_3' \\ \vdots \\ x_n'
  \end{pmatrix}
  =
  \begin{pmatrix}
   1/\sqrt 2 & \sqrt{-1}/\sqrt{2} &&& \\ 
   1/\sqrt 2 & -\sqrt{-1}/\sqrt{2} &&& \\
   &&1 && \\
   &&& \ddots & \\ 
   &&&& 1
  \end{pmatrix}
  \begin{pmatrix}
   x_1 \\ x_2 \\ x_3 \\ \vdots \\ x_n
  \end{pmatrix},
\]
the moment map is written as
\begin{equation*}
  \mu_{T^{n-2}} (x)
  = \lambda \left( 
  \frac{\sqrt{-1} (x_1 \overline{x}_2 - \overline{x}_1 x_2)}
       {\|x\|^2},
  \frac{|x_1 + \sqrt{-1} x_2|^2 +|x_3|^2}
       {\|x\|^2},
  \frac{|x_4|^2}{\|x\|^2},
  \dots,
  \frac{|x_n|^2}{\|x\|^2}
  \right)
\end{equation*}
in terms of $x$-coordinates.
Now we change the coordinates on the torus 
so that the moment map
$(\nu_2, \dots, \nu_{n-1}) : \bP^{n-1} \to \bR^{n-2}$
is given by
\begin{equation*}
 \begin{aligned}
  \nu_2(x) 
   &= \frac{\sqrt{-1} \lambda (x_1 \overline{x}_2 - \overline{x}_1 x_2)}
           {\|x\|^2},\\
  \nu_k (x)
   &= \frac{\lambda \sum_{i=1}^{k} |x_i|^2}{\|x\|^2}
    = \lambda 
    - \frac{\lambda \sum_{i=k+1}^{n} |x_i|^2}{\|x\|^2},
   \quad k=3, \dots, n-1.
  \end{aligned}
\end{equation*}
Then the moment polytope is defined by
\[
  \lambda \ge \nu_{n-1} \ge \dots \ge \nu_3 \ge |\nu_2|.
\]

On the other hand, the fact that $X$ is isomorphic to 
a generalized flag manifold 
$SO(n)/S(O(2) \times O(n-2))$ of type $B$ or $D$
enables us to construct a completely integrable system on $X$ 
by using the same method as in \cite{Guillemin-Sternberg_GCS}.
We recall the construction in more general setting.
Let $X = SO(n,\bC)/P$ be a generalized flag manifold
of type $B$ or $D$.
Hereafter we identify the Lie algebra $\mathfrak{so}(n)$ of $SO(n)$ 
with its dual by an invariant inner product.
Then $X$ can be identified with an adjoint orbit in $\mathfrak{so}(n)$
of a point in the positive Weyl chamber, 
i.e. a matrix of the form
\[
  \begin{pmatrix}
    0 & \lambda_1 \\
    -\lambda_1 & 0 \\
    && \ddots \\
    &&& 0 & \lambda_m \\
    &&& -\lambda_m & 0 \\
    &&&&& 0
  \end{pmatrix}
\]
with $\lambda_1 \ge \dots \ge \lambda_m \ge 0$ 
if $n=2m+1$ is odd, and
\[
  \begin{pmatrix}
    0 & \lambda_1 \\
    -\lambda_1 & 0 \\
    && \ddots \\
    &&& 0 & \lambda_m \\
    &&& -\lambda_m & 0 
  \end{pmatrix}
\]
with $\lambda_1 \ge \dots \ge \lambda_{m-1} \ge |\lambda_m|$ 
if $n=2m$ is even.
For each $k < n$, we regard $SO(k)$ as a subgroup of $SO(n)$ by
\[
  SO(k) \cong \left( \begin{array}{c|c}
             SO(k) & 0 \\
            \hline  0 &1_{n-k}
          \end{array} \right)
  \subset SO(n),
\]
and write the moment map of its action on $X$ as 
$\mu_{SO(k)} : X \to \mathfrak{so}(k)$.
For each $x \in X$, 
let $(\lambda_i^{(k)}(x))_i$ be the intersection of the adjoint
orbit of $\mu_{SO(k)}(x) \in \mathfrak{so}(k)$ with the positive 
Weyl chamber.
Then the collection of the functions $\lambda_i^{(k)}$ 
for $k=2, \dots, n-1$ and $1 \le i \le k/2$
gives a completely integrable system on $X$,
which is called the {\it Gelfand-Cetlin system}.
It is easy to see that its image is a polytope defined by 
\begin{equation}
 \begin{alignedat}{9}
  \lmd 1 && \lmd 2 && \lmd 3 && \bxcdots & \lmd m & \\
  & \llmd {2m}1 && \llmd {2m}2 && \bxcdots & \llmd{2m}{m-1} && \vlmd {2m}m \\
  && \llmd {2m-1}1 && \llmd {2m-1}2 && \bxcdots & \llmd{2m-1}{m-1} &\\
  &&& \llmd {2m-2}1 && \bxcdots & \llmd{2m-2}{m-2} && \vlmd{2m-2}{m-1} \\
  &&&& \bxddots && && \bxvdots \\
  &&&&& \llmd 51 && \llmd 52 & \\
  &&&&&& \llmd 41 && \vlmd 42 \\
  &&&&&&& \llmd 31 & \\
  &&&&&&&& \vlmd 21
 \end{alignedat} 
\label{eq:GC_pattern-B}
\end{equation}
if $n=2m+1$ is odd, and
\begin{equation}
 \begin{alignedat}{9}
  \lmd 1 && \lmd 2 && \bxcdots && \lmd{m-1} && \tlmd m  \\
  & \llmd {2m-1}1 && \llmd {2m-1}2 && \bxcdots && \llmd{2m-1}{m-1} &\\
  && \llmd {2m-2}1 && \bxcdots && \llmd{2m-2}{m-2} && \vlmd{2m-2}{m-1} \\
  &&& \llmd {2m-3}1 && \bxcdots && \llmd{2m-3}{m-2} &\\
  &&&& \bxddots && && \bxvdots \\
  &&&&& \llmd 51 && \llmd 52 & \\
  &&&&&& \llmd 41 && \vlmd 42 \\
  &&&&&&& \llmd 31 & \\
  &&&&&&&& \vlmd 21
  \end{alignedat}
  \label{eq:GC_pattern-D}
\end{equation}
if $n=2m$ is even, where 
\[
  \begin{matrix}
    a && c \\
    & b
  \end{matrix}
\]
means $a \ge b \ge c$.
The polytope defined by
(\ref{eq:GC_pattern-B}) or (\ref{eq:GC_pattern-D})
is called the {\it Gelfand-Cetlin polytope}.

In our situation, $X$ is identified with the orbit of
$(\lambda \ge 0 \ge \dots \ge 0)$,
and hence $\lambda_i^{(k)}$ are zero except for 
$\lambda_1^{(k)}$ for $k=2, \dots, n-1$.
In particular, the Gelfand-Cetlin polytope is defined by
\[
  \lambda \ge \lambda_1^{(n-1)} \ge 
  \dots \ge \lambda_1^{(3)} \ge |\lambda_1^{(2)}|,
\]
which coincides with the moment polytope of $X_0^3$.

To construct a toric degeneration of the Gelfand-Cetlin system,
we introduce a degeneration of $X$ in stages as follows.
For $k=n, n-1, \dots, 3$, the $(n+1-k)$-th stage of 
the degeneration is defined by
\[
  \mathfrak{X}_k = 
  \{ ([x_i], t) \in \bP^{n-1} \times \bC \, | \,
  x_1^2 + \dots + x_{k-1}^2 + t x_k^2 = 0 \}.
\]
Then the fiber $X_{k,t}$ over $t \in \bC$ has actions of
$SO(k)$ and a subtorus
$T^{n-k}=\{(\tau_{k+1}, \dots, \tau_{n})\}$.
Let $\mu_{SO(k)} : \bP^{n-1} \to \mathfrak{so}(k)$ be a moment map
of the $SO(k)$-action on $\bP^{n-1}$.
Note that this is a natural extension of the moment map on $X$.
In terms of the homogeneous coordinates, $\mu_{SO(k)}$ is expressed as
\[
  \mu_{SO(k)} (x) =
  \frac{\sqrt{-1} \lambda}{\|x\|^2}
  \bigl( x_i \overline{x}_j - \overline{x}_i x_j \bigr)_{i,j=1, \dots,k}.
\]
A point in the positive Weyl chamber corresponding to 
the image $\mu_{SO(k)} (x)$ is also denoted by
$(\lambda^{(k)}_1 \ge 0 \ge \dots \ge 0)$.
Note that $\lambda_1^{(2)}$ is given by
\[
  \lambda_1^{(2)} 
  = \frac{\sqrt{-1} \lambda
          (x_1 \overline{x}_2 - \overline{x}_1 x_2)}
         {\|x\|^2},
\]
which coincides with $\nu_2$ on $\bP^{n-1}$.
On the other hand, $\lambda_1^{(k)}$ ($k \ge 3$) satisfies
\begin{equation*}
  \left( \lambda^{(k)}_1 \right)^2
  = - \sum_{1 \le i < j \le k} \lambda^2
      \left( \frac{x_i \overline{x}_j - \overline{x}_i x_j}
                   {\|x\|^2} \right)^2
  = \left( \frac{ \lambda \sum_{i=1}^k |x_i|^2}
                 {\|x\|^2} \right)^2
     - \left| \frac{\lambda \sum_{i=1}^k x_i^2}
                   {\|x\|^2} \right|^2.
\end{equation*}
In particular, $\lambda_1^{(k)}$ coincides with $\nu_k$ on
$X_{k+1,0} = \{ x_1^2+ \dots + x_k^2 = 0 \}$.
Using the same argument as in \cite{Nishinou-Nohara-Ueda_TDGCSPF}, we obtain the following:
\begin{prop}
For each stage, we define
\[
  \Phi_{k} = (\lambda_1^{(2)}, \dots, \lambda_1^{(k-1)},
              \nu_k, \dots, \nu_{n-1})
  : X_{k,t} \longrightarrow \bR^{n-2}.
\]
Then the sequence of $\Phi_k$
together with the gradient-Hamiltonian flow give a toric degeneration
of the Gelfand-Cetlin system.
\end{prop}

\section{Vanishing cycle on the quadric surface}
 \label{sc:vc}

In this section, we consider the family $\mathfrak{X}^3$
in the case of $n=4$.
Note that the embedding 
\[
  X = \bP^1 \times \bP^1 \longrightarrow \bP^3 ,
  \quad 
  ([z_0:z_1],[w_0:w_1]) \longmapsto
  [x_1 : x_2 : x_3 : x_4]
\]
is given by
\begin{align*}
    x_1 &= x_0 y_0 - x_1 y_1,\\
    x_2 &= \sqrt{-1}(x_0 y_1 - x_1 y_0),\\
    x_3 &= x_0 y_1 + x_1 y_0,\\
    x_4 &= \sqrt{-1}(x_0 y_0 + x_1 y_1),
\end{align*}
and then the restriction of the Fubini-Study metric 
$\omega_{\mathrm{FS}}$ to $\bP^1 \times \bP^1$ coincides with 
$p_1^* \omega_{\bP^1} + p_2^* \omega_{\bP^1}$,
where $\omega_{\bP^1}$ is the Fubini-Study metric on $\bP^1$ 
and $p_i: \bP^1 \times \bP^1 \to \bP^1$ is the $i$-th projection
for $i= 1,2$.
The anti-diagonal subset
\[
  L = \{ ([x_0: x_1], 
    [\overline{x}_0 : \overline{x}_1 ] )
    \in \bP^1 \times \bP^1 \, | \, 
    [x_0: x_1] \in \bP^1 \}
\]
is a Lagrangian sphere with respect to this symplectic form,
and its image in $\bP^3$ is given by
\[
  L = \{ [x_1: \dots : x_4] \, | \,
         x_1^2 + \dots + x_4^2 = 0, \, 
         x_1, x_2, x_3 \in \bR, \,
          x_4 \in \sqrt{-1}\bR \, \}.
\]

Recall that the family $\mathfrak{X}^3$ is defined by
the equation $x_1^2 + x_2^2 + x_3^2 + t x_4^2 = 0$.
Hence the fiber $X_t$ over $t \in \bC$ is given by
$X_t = f^{-1}(t)$ for a rational function
\[
  f= - \frac{x_1^2 + x_2^2 + x_3^2}{x_4^2}
  : \bP^3 \longrightarrow \bP^1.
\]

\begin{lemma}
The gradient-Hamiltonian flow of $f$ on $\bP^3$
sends $L$ to the singular point of $X_0$.
\end{lemma}

\begin{proof}
We regard $(x_1, x_2, x_3)$ as a coordinate on 
$\{x_4 = \sqrt{-1} \} \cong \bC^3$ so that $L$ is given by
\[
  L = \{ (x_1, x_2, x_3) \in \bR^3 \, | \,
         x_1^2 + x_2^2 + x_3^2 = 1 \}
    \subset \bC^3.
\]
Since $\omega_{\mathrm{FS}}$, $L$, and $f$ are 
invariant under the natural $SO(3,\bR)$-action,
it suffices to show that the gradient-Hamiltonian
flow starting at $(1,0,0) = [1:0:0:\sqrt{-1}] \in L$ 
goes to the singular point 
$(0,0,0) = [0:0:0:\sqrt{-1}] \in X_0$.
More precisely,
we will show that the orbit of the gradient trajectory
is given by
$\{ x_2 = x_3 = \mathrm{Im}\, x_1 = 0 \}$.

Recall that the gradient-Hamiltonian vector field of $f$
is defined by
\[
  V = - \frac{\nabla (\re f)}{| \nabla (\re f) |^2}
    = \frac{\xi_{\im f}}{|\xi_{\im f}|^2} 
\]
where $\nabla (\re f)$ is the gradient vector field of 
$\re f$, and
$\xi_{\im f}$ is the Hamiltonian vector field
of $\im f$.
The flow of $V$ sends a fiber $X_t$ to another one
(see \cite{Ruan_LTFQHI}).
Since
\[
  \omega_{\mathrm{FS}} =
    \frac{\sqrt{-1}}2 \partial \overline{\partial}
    \log (1 + \| x \|^2 )
\]
is given by
\[
  \omega_{\mathrm{FS}} =
    \frac{\sqrt{-1}}2
    \frac{dx_1 \wedge d \overline{x}_1}
         {(1 + | x_1 |^2)^2}
    + 
    \frac{\sqrt{-1}}2 \sum_{i=2,3}
    \frac{dx_i \wedge d \overline{x}_i}{1 + | x_1 |^2}
\]
on $\{ x_2= x_3 = 0 \}$,
and $df = - 2 \sum x_i d x_i = - 2 x_1 dx_1$ on 
$\{ x_2 = x_3 = 0\}$, the Hamiltonian vector field of 
$\mathrm{Im} f$ has only $\partial / \partial x_1$ and 
$\partial / \partial \overline{x}_1$ components.
This implies that the gradient-Hamiltonian flow preserves 
$\{x_2 = x_3 = 0 \}$.
Next we take be a polar coordinate $x_1 = r e^{\sqrt{-1} \theta}$.
Then we have 
\[
  \omega_{\mathrm{FS}} = \frac{r dr \wedge d \theta}{(1+r^2)^2}
\]
and 
\[
  \mathrm{Im} (df)
  =
  - \frac{2r}{\sqrt{-1}} \bigl( 
   e^{\sqrt{-1} \theta} 
  (dr + \sqrt{-1} r d \theta )
  -
  e^{- \sqrt{-1} \theta} 
  (dr - \sqrt{-1} r d \theta
  ) \bigr)
  = - 4 r^2 d \theta
\]
on $\{\theta = 0 , \, x_2 = x_3 = 0 \}$.
Hence the Hamiltonian vector field of
$\mathrm{Im} f$ is given by
\[
  - 4r (1 + r^2)^2 \frac{\partial}{\partial r},
\] 
which shows the above claim.
\end{proof}

\begin{remark}
$L$ is the fixed locus of the antiholomorphic involution
\[
  \iota : \bP^1 \times \bP^1 \longrightarrow 
  \bP^1 \times \bP^1,
  \quad
  ([z_0:z_1], [w_0:w_1]) \longmapsto
  ([\overline{w}_0:\overline{w}_1], 
    [\overline{z}_0:\overline{z}_1]).
\]
This involution extends to an involution
\[
  \iota : \bP^3  \longrightarrow 
  \bP^3,
  \quad
  [x_1:x_2:x_3:x_4] \longmapsto
  [\overline{x}_1:\overline{x}_2:
  \overline{x}_3:-\overline{x}_4].
\]
Since the Fubini-Study metric and $\mathrm{Re} (f)$ are
$\iota$-invariant,
the gradient vector field $\nabla \mathrm{Re} (f)$ is
also $\iota$-invariant.
In particular, $\nabla \mathrm{Re} (f)$ preserves 
the fixed locus $(\bP^3)^{\iota}$.
This also shows that the gradient-Hamiltonian flow
sends 
\[
  L = (\bP^1 \times \bP^1)^{\iota}
  = X_1 \cap (\bP^3)^{\iota}
\]
to the singularity
$[0:0:0:1] = X_0 \cap (\bP^3)^{\iota}$
of $X_0 = \bP(1,1,2)$.
\end{remark}

Now a result of Fukaya, Oh, Ohta and Ono
shows the following:

\begin{theorem}[{\cite[Theorem 1.9]{FOOO_ASIFC}}]
The $A_\infty$-structure on $H^*(L; \Lambda_0)$
for the anti-diagonal Lagrangian submanifold $L$ satisfies
$
 \frakm_0 = \frakm_1 = 0
$
and $\frakm_2$ can be identified with the quantum cup product
on $H^*(\bP^1; \Lambda_0) \cong H^*(L; \Lambda_0)$.
\end{theorem}

Since the quantum cohomology ring of $\bP^1$ is 
isomorphic to the direct sum of two copies of $\Lambda_0$,
a standard deformation theory argument shows that
the $A_\infty$-structure on $H^*(L; \Lambda)$ is equivalent
to the trivial one,
so that the idempotent completion of the full triangulated subcategory
of the derived Fukaya category of $\bP^1 \times \bP^1$
generated by $L$ is equivalent
to the direct sum of two copies of the derived category of $\Lambda$-modules.


\section{Deformation of singular disks in toric weak Fano surfaces}
 \label{sc:f2}

We develop a theory concerning the deformability of 
 singular disks in some toric surfaces.
Consider a degenerating family $\pi: \mathfrak X\to \Bbb C$
 of $\Bbb P^1\times \Bbb P^1$ to
 the Hirzebruch surface $F_2$. 
The surface $F_2$ has the unique $(-2)$-curve, which we write as $D$.
Since $F_2$ is a toric surface, fixing some invariant K\"ahler form $\omega$, 
 the moment map $\mu: F_2\to \Delta$ is defined, where
 $\Delta$ is a suitable trapezoid.
Take a point $u$ in the interior of $\Delta$ and let $L(u)$ be the fiber of $\mu$
 over $u$.

Since the family $\mathfrak X$ is the product $F_2\times\Bbb C$ as a
 differentiable manifold, $\omega$ together with the standard K\"ahler form
 on $\Bbb C$ defines a K\"ahler form on $\mathfrak X$, making
 $\mathfrak X\to \Bbb C$ a toric degeneration of integrable systems
 (the structure of integrable systems on the fibers over $q\in\Bbb C\setminus\{0\}$
 is defined by the pullback of that of $F_2$ by the gradient-Hamiltonian flow
 along some fixed path $\gamma$,
 which gives a toric degeneration of the Gelfand-Cetlin 
 system on $\bP^1 \times \bP^1$).                       
Let $\gamma: [0, 1]\to \Bbb C$, $\gamma(1) = 0$, and let
 $X_s$ be the fiber of $\pi$ over $\gamma(s)$.

There is a unique family of disks of Maslov index two:
\[
(D^2, S^1)\to (F_2, L(u))
\]
 which intersect the curve $D$.
Let $\alpha\in H_2(F_2, L(u);\Bbb Z)$ be the homology class of their images.
We consider stable maps
\[
\phi: (C, S^1)\to (F_2, L(u)) 
\]
 with the following properties.
\begin{itemize}
\item $(C, S^1)$ is a prestable bordered Riemann surface of genus zero, one
 boundary component and without marked point.
\item The homology class of the image is $\alpha+ k[D]$, $k\geq 0$.
\end{itemize}
Pulling back $L(u)$ by the gradient Hamiltonian flow along $\gamma$, 
 we have a family of Lagrangian tori $\{L(u)_{s}\}$, $s\in [0, 1]$, $L(u)_{1} = L(u)$.
The problem we consider is to determine when the stable map 
 $\phi$ can be lifted to a stable map to $(X_s, L(u)_{s})$. 

Let us first consider the case with $k = 0$.
The deformation of $\phi$ is controlled 
 by the tangent-obstruction complex as usual, provided the relevant sheaves
 are replaced by Riemann-Hilbert sheaves.
Such a deformation theory is developed in \cite{Nishinou_DCTVTC}.
We refer to it for precise 
 arguments, and here we omit the technicalities. 
In the case $k = 0$, the relevant complex in $F_2$ is,
\[
0\to (\Theta_{D^2}, \Theta_{S^1})\to 
 \phi^*(\Theta_{F_2}, \Theta_{L(u)})\to \mathcal N\to 0.
\]
Here one easily sees that the normal bundle $\mathcal N$ is isomorphic to the
 trivial Riemann-Hilbert line bundle
 $(\mathcal O_{D^2}, \mathcal O_{S^1})$.

We try to lift $\phi$ to generic fibers order by order.
The obstruction to the lift at each step is controlled by
\[
H^1(D^2, S^1; \mathcal N) = 0,
\]
 so the lifts exist in this case.
The lifts are locally parametrized by 
\[
H^0(D^2, S^1; \mathcal N) = \Bbb R,
\]
 and clearly the global moduli of the lift is isomorphic to that of $\phi$.
Namely, they are parametrized by $S^1$, and the boundary sweeps the fiber
 $L(u)_{s}$ as we move the maps.

Now we move to  the cases with $k>0$.
Note that the Maslov indices of the maps are two for all $k$.

Consider the case with $k = 1$.
In this case, the domain curve is the one point union $C = D^2\cup\Bbb P^1$.
So the moduli of these maps are the same as the case with $k = 0$.

The tangent-obstruction complex for $\phi$, which do not change the structure of
 the singularity of $C$ is, 
\[
0\to \Theta'_C\to \phi^*\Theta_{F_2}\to \mathcal N'\to 0, 
\]
 here $\Theta'_C$ is the sheaf of tangent vector fields on $C$,
see \cite{Kontsevich_ERCTA}.
To include the deformation which smooth the singularity of $C$,
 it suffices to replace $\Theta'_C$ by the logarithmic tangent bundle
 $\Theta_C$.
Then the tangent-obstruction complex becomes
\[
0\to \Theta_C\to \phi^*\Theta_{F_2}\to \mathcal N\to 0.
\]
Note that $\Theta_C$ is an invertible sheaf which restricts to 
 $\mathcal O(1)$ on the component $\Bbb P^1$ of $C$, 
 and to the Riemann-Hilbert bundle $\mathcal E$ on $D^2$,
 such that the doubling of $\mathcal E$ is isomorphic to 
 $\mathcal O_{\Bbb P^1}(-1)$ (the isomorphism class of $\mathcal E$ is
 uniquely determined by this condition).

In this case, one sees that $\mathcal N$ is a line bundle, 
 which restricts to $\mathcal O(-1)$ on $\Bbb P^1$ and 
 to the Riemann-Hilbert bundle $\mathcal E'$ on $D^2$,
 such that the doubling of $\mathcal E'$ is isomorphic to 
 $\mathcal O_{\Bbb P^1}(1)$.
Then it is easy to see that the obstruction
\[
H^1(C, S^1;\mathcal N)
\]
 vanishes, so that $\phi^*$ lifts in this case, too.

However, since there is no $(-2)$-curve in $\Bbb P^1\times\Bbb P^1$, 
 the singular point of $C$ has to be smoothed, so that the lifts are maps from 
 the disk.

The tangent space of the moduli is
\[
H^0(C, S^1;\mathcal N)\cong \Bbb R,
\]
 as in the case with $k = 0$.
The global moduli is also the same as this case.

Finally, we consider the cases with $k>1$.
As we remarked above, when there is a lift of $\phi$, the lift
 has to be a map from the disk.
Let us assume there is such a lift $\psi_k$.
On the other hand, let $\psi_1$ be a lift of $\phi$ for $k = 1$ constructed above.
Then the intersection number of the images of these maps 
 can be calculated in $F_2$ as follows:
\[
((\phi|_{D^2})_*[D^2]+k[D])\cdot ((\phi|_{D^2})_*[D^2]+[D])
 = (k+1)+k\cdot(-2)
 = 1-k<0.
\]
However, the lifts $\psi_k$ and $\psi_1$ are different holomorphic curves, 
 and it is impossible to have negative (local) intersection.
So there is no lift when $k>1$.
This classifies the possibilities of the existence of lifts of stable disks for the case of 
 $F_2$.
 
The case of (smooth) toric degeneration of cubic surfaces can be 
 calculated with minor changes.
The central fiber (toric surface) has nine boundary components, 
 which is a cycle of rational curves with self intersection numbers
 \[
 -1, -2, -2, -1, -2, -2, -1, -2, -2.
\]
Let $(D', D'')$ be one of the three pairs of neighboring $(-2)$-curves. 
The liftable stable disks with singular domain curves are embeddings whose images are
 classified as follows:
\[
D^2\cup D', \;\; D^2\cup D'',\;\; D^2\cup D'\cup D'',
\]
 where the pair $(D', D'')$ can be any of the three.
Note that the $D^2\cup D'\cup D''$ case has two possibilities, according to which 
 of the $(-2)$-curves the disk intersects.

{\bf Acknowledgment}:
T.~N. is supported by Grant-in-Aid for Young Scientists (No.19740034).
Y.~N. is supported by Grant-in-Aid for Young Scientists (No.19740025).
K.~U. is supported by Grant-in-Aid for Young Scientists (No.18840029).
K.~U. also thanks the Mathematical Institute at the University of Oxford
for hospitality and Engineering and Physical Sciences Research Council
for financial support.


\bibliographystyle{plain}
\bibliography{bibs}


\end{document}